\documentclass[12pt]{article}
\usepackage{amsmath}
\usepackage{amssymb}
\usepackage{latexsym}
\usepackage{enumerate}
\usepackage[ansinew]{inputenc}
\usepackage[all]{xy}
\usepackage[usenames]{color}

\newtheorem{theorem}{Theorem}[section]
\newtheorem{proposition}[theorem]{Proposition}
\newtheorem{corollary}[theorem]{Corollary}
\newtheorem{lemma}[theorem]{Lemma}
\newtheorem{example}[theorem]{Example}
\newtheorem{remark}[theorem]{Remark}
\newtheorem{question}[theorem]{Question}
\newtheorem{definition}[theorem]{Definition}
\newenvironment{proof}{\noindent{\sc Proof.}}{\quad\qed\medskip}

\newcommand{\Z}{{\mathbb Z}}

\newcommand{\C}{{\cal C}}

\newcommand{\Hom}{{\rm Hom}}

\newcommand{\Ext}{{\rm Ext}}
\newcommand{\Ker}{{\rm Ker}\,}

\newcommand{\colim}{{\rm colim}\,}

\newcommand{\qed}{\quad\lower0.05cm\hbox{$\Box$}}

\newcommand{\downarrowright}[1]{\downarrow
\rlap{\raise0.1cm\hbox{$\scriptstyle{#1}$}}}

\newcommand{\downarrowleft}[1]{\rlap{\kern-0.2cm
\raise0.1cm\hbox{$\scriptstyle{#1}$}}\downarrow}

\newcommand{\uparrowright}[1]{\uparrow
\rlap{\lower0.1cm\hbox{$\scriptstyle{#1}$}}}

\newcommand{\uparrowleft}[1]{\rlap{\kern-0.2cm
\lower0.1cm\hbox{$\scriptstyle{#1}$}}\uparrow}

\newcommand{\epi}{\mbox{$\to$\hspace{-0.35cm}$\to$}}

\newcommand{\Zup}{\Z[1/p]}

\newcommand{\Cof}{{\rm Cof}}

\begin{document}
\setcounter{page}{1}
\title{Generators and closed classes of groups}
\footnotetext{


The first author was partially supported by FEDER-MEC grants MTM2010-20692 and MTM-2016-76453-C2-1-P. Second-named author was supported by grant MTM2016-76453-C2-2-P, and grant FQM-213 of the Junta de Andaluc\'{\i}a.}

\author{Ram\'on Flores and Jos\'e L. Rodr\'{\i}guez}
\date{September 30th, 2020}

\maketitle

\begin{abstract}
We show that in the category of groups, every singly-generated class which is closed under isomorphisms,
direct limits and extensions is also singly-generated under isomorphisms and
direct limits, and in particular is co-reflective. In this way, we extend to the group-theoretic framework
the topological analogue proved by Chach\'olski, Parent and Stanley in 2004. We also establish several new relations
between singly-generated closed classes.

\end{abstract}


\setcounter{equation}{0}

\section{Introduction}

 In a category of $R$-modules over a certain ring $R$, the study of the closure of different classes under different operations has been a hot research topic for decades (see \cite{GT06} for a thorough approach). The operations include extensions, subgroups, quotients, etc. In the case of $R=\mathbb{Z}$, we are dealing, of course, with operations in the category of abelian groups.

There is another different but related problem. Given a collection $\mathbf{D}$ of objects in a class $\mathcal{C}$, the smallest class that contains $\mathbf{D}$ and is closed under certain operations is called the class generated by $\mathbf{D}$ under these operations. When this class is equal to $\mathcal{C}$, it is said that $\mathcal{C}$ is \emph{generated} by $\mathbf{D}$ under the operations, and the elements of $\mathbf{D}$ are called \emph{generators} of $\mathcal{C}$. In this situation, if $\mathbf{D}$ consists of only one object, the class $\mathcal{C}$ is said to be \emph{singly-generated} under the operations. Now the question is clear: given a class $\mathcal{C}$ defined by a certain property, is it possible to construct a generator for it? A standard reference for the topic is \cite{Adamek}, and related problems have been recently treated in \cite{CGR11}.

In this paper, we will deal with different versions of these two problems in the category of groups. Rather than from the vast literature in Module and Category Theory related to these topics, our motivation and strategy come from the field of Homotopy Theory. From the nineties, Bousfield, Farjoun and Chach\'olski (\cite{Bou94}, \cite{Cha96}, \cite{F97}) among others, developed an ambitious program whose main goal was to classify spaces and spectra in terms of classes defined by certain operations, particularly homotopy direct limits (cellular classes) and fibrations (periodicity classes). Taking account of this framework, the fundamental group functor gives a natural (but non-trivial) way to understand closed classes of groups from the same viewpoint.

We are mainly interested in three classes of operations in the class of groups, namely taking direct limits, constructing extensions from a given kernel and cokernel, and taking quotients. The relation between the closed classes under these operations has been already studied in \cite{Cha14} from a different point of view. We focus on singly-generated classes, and in fact our first main result is a generation result. Given a group $A$, we denote by $\mathcal{C}(A)$ the class generated by $A$ under direct limits and isomorphisms, and by $\overline{\mathcal{C}(A)}$ the class generated by $A$ under direct limits, extensions and isomorphisms. The elements of $\mathcal{C}(A)$ are usually called $A$-\emph{cellular} groups, and the elements of $\overline{\mathcal{C}(A)}$ are called $A$-\emph{acyclic} groups. Then:

\vspace{.5cm}

\noindent
\textbf{Theorem \ref{maintheorem}} {\it For any group $A$, there exists a group $B$ such that $\overline{\mathcal{C}(A)}=\mathcal{C}(B)$}.

\vspace{.3cm}
Moreover, we are able to describe $B$ in terms of groups of the class $\overline{\mathcal{C}(A)}$. This statement allows to generalize results of \cite{RS01}, where the case in which the Schur multiplier of $A$ is trivial was studied, and also of \cite{BC99} for the case of groups with trivial ordinary homology. It can also be seen as the group-theoretic version of the main theorem of \cite{CP04}, and it is in the spirit (see last section) of the generators of radicals described in \cite{CRS99}. We remark that the proof of \cite{CP04} cannot be translated to the category of groups, because the simplicial structure of the category of spaces is crucial in it.

Section~4 is devoted to establish relations between the closed classes under the operations mentioned above. The hierarchy is the following:
$$
\xymatrix{\mathcal{C}(A) \ar@{^{(}->}[r] \ar@{^{(}->}[d] &  \overline{\mathcal{C}(A)} \ar@{^{(}->}[d] \\
\mathcal{C}_q(A) \ar@{^{(}->}[r] &  \mathcal{C}_t(A) }
$$
where $\mathcal{C}_q(A)$ is the closure of $\mathcal{C}(A)$ under quotients and $\mathcal{C}_t(A)$ is the closure of $C(A)$ under quotients and extensions. All these classes turn out to be co-reflective. Furthermore, we find an example where the four classes are distinct (see Example~\ref{alldiferent}).

The co-reflections associated to the classes $\mathcal{C}(A)$, $\mathcal{C}_q(A)$ and $\mathcal{C}_t(A)$ are given respectively by the $A$-cellular cover $C_A$, the $A$-socle $S_A$ and the $A$-radical $T_A$, which are well-known notions in Group/Module Theory (see definitions in Section \ref{background}). However, for the class $\overline{\mathcal{C}(A)}$, the existence of the corresponding co-reflection $D_A$ had been previously established \cite{RS01} only for groups for which there exists a two-dimensional Moore space, and this goal was achieved only by topological methods. This gap in the theory is filled in this paper making use of Theorem \ref{maintheorem}, that guarantees the existence of $D_A$ for any group $A$. Roughly speaking, the philosophy is the following: given a group $G$, a singly-generated class $\mathcal{C}$ and the co-reflection $F$ associated to a certain operation, the kernel and the cokernel of the homomorphism $F(G)\rightarrow G$ measure how far $G$ is from belonging to the class generated from $\mathcal{C}$ under the operation. We prove how the co-reflections define the corresponding closed classes, describe when possible the kernels, and discuss, for different instances of the generator, the (in)equality between the corresponding classes. In particular, we prove the following:

\vspace{.3cm}

\noindent
\textbf{Proposition \ref{TASA}}
{\it Suppose that for a group $A$ the class $\C(A)$ is closed under extensions, i.e. $\mathcal{C}(A)=\overline{\mathcal{C}(A)}$. Then $\mathcal{C}_q(A)=\mathcal{C}_t(A)$.
}
\vspace{.3cm}

\noindent
The converse of this result (stated as Question \ref{SoT}) is not true in general. Section \ref{Burnside} is devoted to give a counterexample that involves the Burnside idempotent radical and the Burnside group $B(2,q)$. It is also an interesting question to find groups for which this converse is true. In fact, we are not even aware of an example of any example of abelian group $A$ such that $\mathcal{C}_q(A)=\mathcal{C}_t(A)$, but $\mathcal{C}(A)\neq\overline{\mathcal{C}(A)}$ (or equivalently $S_A=T_A$ and $C_A\neq D_A$); see Proposition \ref{additive} and the subsequent discussion.

All the classes that appear in this paper will be supposed to be closed under isomorphisms. Moreover, the axioms of ZFC will be assumed throughout the paper.


\section{(Co)reflections}
\label{background}

\subsection{Background on co-reflections}

In this section we introduce the main notions that will be used throughout the rest of our paper. We will start with the classical categorical definition of \emph{co-reflection}, that will be crucial for us.


\begin{definition}

{\rm
Given an inclusion of categories ${\bf C'}\subseteq {\bf C}$, where ${\bf C'}$ is full in ${\bf C}$, ${\bf C'}$ is \emph{co-reflective} in ${\bf C}$ if the inclusion functor possesses a right adjoint $(F,\mu)$, called \emph{co-reflection}.
}
\end{definition}

A good survey about the topic from the categorical point of view can be found in \cite{Adamek}. We recall here some basic facts and notations that will be used later.

	{\rm
	\begin{enumerate}
	\item Given a co-reflective subcategory ${\bf C'}\subseteq {\bf C}$, the associated co-reflection $F: {\bf C}\to {\bf C'}$ is unique up to natural isomorphism.
	\item Objects $X$ satisfying $FX\cong X$ are called {\it $F$-colocal}, and the morphisms $X\to Y$ inducing isomorphisms $FX\to FY$, are called {\it $F$-colocal equivalences}.
	\item An object $Z$ is $F$-colocal if and only if ${\bf C}(Z, X)\cong {\bf C}(Z,Y)$ for all $F$-colocal equivalences $X\to Y$, and also, a morphism $X\to Y$ is an $F$-colocal equivalence if and only if ${\bf C} (Z,X)\cong {\bf C}(Z,Y)$ of all $F$-colocal objects $Z$. In particular, $\eta: FX\to X$ is an $F$-colocal equivalence, and $FX$ is $F$-colocal.
	\item An inclusion ${\bf C}_1 \subseteq {\bf C}_2$ of two co-reflective subcategories induces a natural transformation $\tau: F_1 \to F_2$ between the corresponding co-reflections, and viceversa. Furthermore, $\tau$ is unique up to natural equivalence.
	\end{enumerate}
}

It is well-known that co-reflections preserve finite products in the categories of modules and spaces. Next we provide an analogous statement for arbitrary groups, as we are not aware of a proof of it in the literature (see \cite[Lemma 2.3]{FGS07} for a concrete instance).

\begin{proposition}
\label{coreflection-product}
Let $(F,\mu)$ with $F:Groups \to \mathcal{C}$ a co-reflection from the category of groups. Then, for every pair of groups $G$ and $H$, $F(G\times H)\cong FG\times FH$.
\end{proposition}
\begin{proof}
By Statement 3) above,   $\mu_G\times \mu_H: FG \times FH \to G\times H$ is an $F$-colocal equivalence, as $$\Hom(Z, FG\times FH) \simeq\Hom(Z,FG)\times \Hom(Z,FH)\simeq\Hom(Z,G)\times \Hom(Z,H)\simeq \Hom(Z, G\times H),$$
for all $F$-colocal groups $Z$, and the composition is induced by $\mu_G\times \mu_H$.
To prove that $FG\times FH$ is $F$-colocal, set an ordering $FH=\{x_i: i<\alpha\}$ for some ordinal~$\alpha$, and consider the linear sequence
$$
FG * FH \stackrel{f_0}{\longrightarrow} FG* FH \stackrel{f_1}{\longrightarrow}
\cdots
\longrightarrow FG* FH \stackrel{f_i}{\longrightarrow}
FG* FH \stackrel{f_{i+1}}{\longrightarrow}
\cdots$$
where $f_i(g)=x_i g x_i^{-1}$ for $g\in FG$, and $f_i(h)=h$, for $h\in FH$.
The colimit is isomorphic to the product $FG\times FH$. Thus $FG\times FH$ is a direct limit of of $F$-colocal groups, and hence $F$-colocal. Uniqueness, up to isomorphism, yields $F(G\times H)\cong FG\times FH$.
\end{proof}

This property will be useful in particular in Example~\ref{alldiferent}.
\subsection{Group-theoretic constructions}
The following are well-known examples of group-theoretic constructions. As we shall see in Section \ref{closedclasses}, they give rise to different co-reflections over interesting (co-reflective) subcategories of the category of groups.

\begin{definition}
{\rm
Let $A$ and $H$ be groups.

\begin{enumerate}
\item The $A$-\emph{socle} $S_A H$ is the subgroup of $H$ generated by the images of all the homomorphisms $A\rightarrow H$.


\item The $A$-\emph{radical} $T_A H$ of $H$ is the normal subgroup of $H$ characterized by the following property:  $H/T_A H$ is the largest quotient of $H$ such that $\Hom(A,H/T_A H)=0$.
Groups $G$ that satisfy $\Hom(A,G)=0$ are called {\it $A$-reduced} (or {\it $A$-null}). The projection $\varphi: H\epi H/T_AH$ is called {\it $A$-reduction} (or {\it $A$-nullification}), where $H/T_AH$ is $A$-reduced and $\varphi$  induces a bijection $\Hom(H/T_AH,G)\simeq \Hom(H,G)$, for all $A$-reduced groups $G$.
\item The $A$-\emph{cellular cover} of $H$ is the unique group $C_A H\in \C(A)$, up to isomorphism, that can be constructed from copies of $A$ by iterated direct limits, and such that there exists a homomorphism $C_A H\rightarrow H$ which induces a bijection $\Hom(A,C_A H)\simeq \Hom(A,H)$ (an \emph{A-equivalence}). Observe that this is equivalent to $\Hom(G,C_A H)\simeq \Hom(G,H)$ being a bijection for every $G\in \C(A)$.
\end{enumerate}
}
\end{definition}

\begin{remark}
{\rm
Note that similar definitions are possible if we change $A$ by a collection $\mathcal{A}$ of groups indexed over a certain set $I$. In this case, the generator would be the free product of isomorphism classes of elements of $\mathcal{A}$. Moreover, it is possible to give versions of the socle and the radical relative to a class of epimorphisms of groups. See \cite{Adamek} for a general treatment of these topics. On the other hand, $A$-reduction is a particular example of localization functor $L_f$, with respect to the constant homomorphism $A\rightarrow \{1\}$ (see for example \cite{Cas00} for background on localization of groups and spaces).
}
\end{remark}


 For the convenience of the reader, we recall more constructive versions of the radical and the cellular cover. The  $A$-radical $T_A H$ is the union
$T_{A}(H)= \bigcup_i T^i(H)$ of a chain starting at
$T^0(H)=S_{A}(H)$ and defined inductively by $T^{i+1}(H)/T^i(H)=S_{A}(H/T^i(H))$,
and $T^\lambda(H)=\bigcup_\alpha  T^\alpha(H)$ if $\lambda$ is a limit ordinal. This process yields an $A$-reduced group $H/T_AH$. Furthermore, if $\varphi: H\to G$
is any homomorphism with $G$ $A$-reduced, then it factors through $H\to H/T_AH \to G$. To see this, observe that if
$S_AH \to H\to G$ is trivial, $\varphi$ factors through $H\to H/S_A(H)\to G$. Inductively, if $i+1$ is a successor ordinal, since $S_A(H/T^i H)\to H\to G$ trivial, then $\varphi$ factors through $H/T^{i+1}H\to G$. If $\lambda$ is an ordinal limit, $H/T^\alpha H \to G$ for $\alpha< \lambda$, and this yields $H/T^\lambda H \to G$.

In turn, the $A$-cellular cover $C_{A}H$ is constructed in \cite[Theorem 1]{CDFS08} in the following explicit way. Let $A$ and $H$ be two groups, $C^H_0= \star_{A\to H}H$ the free product indexed over all homomorphisms $A\rightarrow H$, $ev: C^H_0\to H$ the evaluation homomorphism, and $K^H_0$ the kernel of the evaluation. Then,
$$
C^H_A(H) \cong (C^H_0/[C^H_0,K^H_0]) / T_A(K^H_0/[C^H_0,K^H_0]).
$$

An inductive description of $C_{A}H$ can be found in Section 3 of \cite{RS01}.

Next we will check that the previous constructions are functorial, and then they provide an important source of co-reflections (see Section~4):


\begin{proposition}

Given a group $A$, the $A$-socle, the $A$-radical and the $A$-cellular cover define functors in the category of groups.

\end{proposition}

\begin{proof}
Consider a group homomorphism $f:G\rightarrow H$. To check functoriality we should verify that $f$ restricts to the corresponding $A$-socles, $f:S_AG \to S_AH$. Let $x$ be an element of $S_AG$. Then there exist a finite set of homomorphisms $\phi_i:A\rightarrow G$, $i\leq n$, and elements $x_i\in A$ such that $x=\phi_1(x_1)\ldots \phi_n(x_n)$. Hence, $f(x)=(f\circ\phi_1(x_1))\ldots (f\circ\phi_1(x_n))$, which implies $f(x)\in S_AH$.

By induction along the increasing chain $T_{A} H= \bigcup_i T^i H$, one can prove functoriality of all $T^i$ and hence $T_A$. Let $f:G\to H$ be a group homomorphism. If $i+1$ is a successor ordinal, and assuming that $f$ restricts to $T^i H \to T^i G$, it is obtained that $f$ induces a homomorphism on the quotients $ H/T^iH \to G/T^iG$. Now the functoriality of $S_G$ yields the restriction $S_A(H/T^i H)\to S_A(G/T^i G)$, and therefore to $T^{i+1}H\to T^{i+1}G$.
For a limit ordinal $\lambda$, if $f$ restricts to $T^\alpha H\to T^\alpha G$ for each $\alpha<\lambda$, then is can be extended to $T^\lambda H\to T^\lambda G$.
Note that the previous argument also yields the functoriality of the quotients $H\to H/T^iH$ and hence that of the $A$-reduction homomorphism $H\to H/T_AH$.

To check functoriality of the $A$-cellular cover, consider a homomorphism $f:H\rightarrow H'$. Following the notation of the previous construction of $C_AH$, $f$ induces a homomorphism $\tilde{f}:C_0^H\rightarrow C_0^{H'}$ and a commutative diagram

$$\xymatrix{K_0^H \ar[r] \ar[d] & C_0^H \ar[r] \ar[d]^{\tilde{f}}  & H \ar[d] \\
K_0^{H'} \ar[r] & C_0^{H'} \ar[r]  & H' }
$$ whose rows are group extensions. Now, as $\tilde{f}(K_0^H )\subseteq K_0^{H'}$, it is clear that $\tilde{f}$ (and then $f$) induce a homomorphism $C^H_0/[C^H_0,K^H_0]\rightarrow C^{H'}_0/[C^{'H}_0,K^{H'}_0]$. Now the result follows from the functoriality of the radical, taking account of the previous description of $C_A(H)$.
\end{proof}

\begin{remark}
{\rm
  A thorough analysis of cellular covers and its functorial behaviour is undertaken in Section 2 of \cite{BCFS13}.
}
\end{remark}

 Concerning cellular covers, a different characterization was obtained by J. Scherer and the second author:

\begin{theorem}[\cite{RS01}, Theorem 3.7]
\label{CAH} Given groups $A$ and $H$, there is a central extension
\begin{equation}
\label{KCS}
K \hookrightarrow C_A(H) \twoheadrightarrow S_A(H),
\end{equation}
such that $\Hom(A_{\rm ab},K)=0$ and the natural map $\Hom(A,H)\to H^2(A;K)$ is
trivial. Moreover, this extension is initial with respect to all the extensions with these two properties.
\end{theorem}

\subsection{Covers of spaces}

We will also need to review some constructions in the category of \emph{spaces}. Analogously to groups, given a pointed space $M$, the \emph{cellular class} $\mathcal{C}(M)$ generated by $M$ is the smallest class that contains $M$ and is closed under homotopy equivalences and pointed homotopy colimits. The corresponding augmentation $CW_M X\to X$ for a pointed space $X$, called the $M$-\emph{cellular cover} of $X$, was explicitly described by Chach\'olski in his thesis \cite{Cha96} in the following way: consider the homotopy cofiber sequence
\begin{equation}
    \label{cof1}
    \bigvee_{[M,X]_*} M\stackrel{ev}{\longrightarrow} X \to \Cof_MX
\end{equation}
where $ev$ is the evaluation map.
Denote by $P_M$ the $M$-nullification functor, i.e. the homotopical localization with respect to the constant map $M\rightarrow \ast$ (see \cite{F97}, chapter 1). Then for any connected space $X$, $CW_MX$ fits into the following homotopy fiber sequence
\begin{equation}
    \label{fib1}
    CW_MX\rightarrow X \rightarrow P_{\Sigma M}(\Cof_M X).
\end{equation}
In \cite[Theorem 3.3]{RS01} the relation between cellular covers of groups and spaces was described: if $M$ is a 2-dimensional complex with $\pi_1M=A$, and $X$ is an Eilenberg\-Mac Lane space  $K(H,1)$ then $\pi_1CW_MX=C_AH$ for every group $H$.

\section{Generating classes of groups}
\label{generating}

 Given a group $A$, the class $\overline{\mathcal{C}(A)}$ is the smallest class that contains $A$ and is closed under direct limits and extensions. The elements of this class will be called $A$-\emph{acyclic} groups. The main goal of this section will be to show that such a class is always singly-generated using only direct limits, or in other words, it possesses a \emph{cellular generator}:

\begin{definition}
\rm{ Given a class $\mathcal{C}$ of groups closed under direct limits, a \emph{cellular generator} for the class $\mathcal{C}$ is a group $B$ such that $\mathcal{C}=\mathcal{C}(B)$.

}\end{definition}

In our context, $\mathcal{C}$ will usually be the class $\overline{\mathcal{C}(A)}$ for some group $A$. These cellular generators are not unique, and in fact several models for particular instances of $A$ have been proposed in the literature (\cite{BC99}, \cite{RS01}). We now present a general construction, which is inspired in the methodology of Chach\'olski-Parent-Stanley \cite{CP04}, that solved successfully the same problem in the category of topological spaces. Concretely, we adopt the group-theoretic context their idea of small space and follow the general scheme of their proof; however, it should be remarked that the translation is highly non-trivial, as their methods are homotopical in nature and rely in the simplicial structure of space, and we deal with a complete different context.

 We start with a set-theoretic definition. As usual, we denote the cardinal of a set $X$ by $|X|$.




\begin{definition}
\rm{For a given infinite regular uncountable cardinal $\lambda$ a group $H$ is called {\it $\lambda$-small}
(or small with respect to $\lambda$) if $|H| < \lambda$.}
\end{definition}

In the sequel, unless explicit mention against, we will assume that $\lambda$ is regular infinite.
\vspace{.5cm}
Next we show the easiest example of $\lambda$-smallness.

\begin{example}
{\rm For $\lambda=\aleph_1$, $\lambda$-small means finite or countably infinite.}
\end{example}

Using the previous definition, we can define the class $\C[\lambda]$;

\begin{definition}
	{\rm
Given a class $\C$ of groups closed under direct limits,
we denote by $\C[\lambda]$ the subclass of $\C$ defined by the following property: $H$ belongs to $\C[\lambda]$
if for each $\lambda$-small group $K$, every group homomorphism $K \to H$
factors as $K \to C\to H$ for some $\lambda$-small group $C$ in $\C$.
}
\end{definition}

Note that the definition of $\C[\lambda]$ only depends on $\C$ and the cardinal $\lambda$. We remark that the group $C$ depends on the homomorphism $K\rightarrow H$ (and in particular depends on $K$ and $H$).

The following characterization says that we can restrict ourselves to subgroups of $H$.

\begin{lemma}
Let $\lambda$ be an infinite regular cardinal. Then a group $H$ is in $\C[\lambda]$ if and only for each $\lambda$-small subgroup $S<H$, there exists a $\lambda$-small subgroup $C$ in $\C$, such that $S<C<H$.
\end{lemma}
\begin{proof}
Let $K \to H$ be any group homomorphism, with $K$ $\lambda$-small. The image $\varphi(K)$ is $\lambda$-small subgroup of $H$. Assuming the necessary condition, there exists a $\lambda$-small subgroup $C\in \C$ such that $\varphi(K)<C<H$. Hence, $K\to \varphi(K)\to C \to H$ is the desired factorization.

For the converse, suppose that $S$ is any $\lambda$-small subgroup of $H$, and let  $S \to C_0 \stackrel{\phi_0}{\to} H$ be a factorization of the inclusion $S\hookrightarrow H$, where $C_0$ is $\lambda$-small and belongs to $\C$. Then $\phi_0(C_0)<H$ is $\lambda$-small, and $S\leq \phi_0(C_0)$. Denote $K_0=\phi_0(C_0)$. If $S=K_0$, then $S$ is a retract of $C_0$, and hence it belongs to $\C$. Therefore, we can take $C=S$ and we are done. If $S<K_0$, we can proceed in the same way with the inclusion $K_0 <H$, obtaining a factorization $K_0\to C_1 \stackrel{\phi_1}{\to} H$. Iterating the process if necessary, we obtain a string of homomorphisms such that for every $i$ there is a strict inclusion $K_i\lneq K_{i+1}$:
$$S\hookrightarrow C_0\twoheadrightarrow K_0\hookrightarrow C_1\twoheadrightarrow K_1\hookrightarrow C_2\twoheadrightarrow K_2\hookrightarrow\ldots,$$
where $C_i$ is $\lambda$-small in $\C$ for every $i$.
Passing to the countable limit we get that the composition $$\colim_{i<\omega} K_i \to \colim_{i<\omega} C_{i+1} \to \colim_{i<\omega} K_{i+1},$$ is the identity, so $\colim K_i$ is a retract of $\colim C_i$, and thus belongs to $\C$. Thus, $C=\colim K_i$ is the desired subgroup of $H$.
\end{proof}

The following proposition determines the scope of the class $\C[\lambda]$.

\begin{proposition}
\label{lambda-closed}
Let $\lambda$ be an infinite regular cardinal. A group $H$ is in $\C[\lambda]$ if and only if
it is isomorphic to a direct limit of $\lambda$-small groups in $\C$.
\end{proposition}
\begin{proof} By the previous proposition, for each $\lambda$-small subgroup $S$ of $H$, there exists a $\lambda$-small subgroup in $\C$ such that  $S<C<H$. Hence, $H$ is isomorphic to the direct system of all $\lambda$-small subgroups of $H$ that belong to $\mathcal{C}$.

 Conversely, let $H$ be a direct limit of $\lambda$-small groups in $\mathcal{C}$, and let us prove that $H$ belongs to $\mathcal{C}[\lambda]$. Following (\cite{Cha94}, 4.3), and taking into account that $\pi_1$ commutes with colimits, every direct limit of groups can be obtained as an iteration of limits over linear categories (chains) and push-outs, so we need to check the statement in these two cases. Observe that $H$ belongs to~$\mathcal{C}$.

    Let $H$ be the direct limit of the $\lambda$-small groups $\{G_{\alpha}\}$ over a linear category $\mathcal{I}$, i.e. $G_1\rightarrow G_2\rightarrow \ldots \rightarrow H$, $\kappa$-indexed for a certain ordinal $\kappa$. As $H$ is a regular cardinal, it makes sense to distinguish between the following two cases. If $\mathcal{I}$ is $\lambda$-filtered, a standard compactness argument states that every homomorphism $K\rightarrow H$ from a $\lambda$-small group factors through $G_\alpha$ for a certain $\alpha$. On the other hand, if $\mathcal{I}$ is not $\lambda$-filtered, this means that the cardinality of $\kappa$ is smaller than $\lambda$, and then the diagram is $\lambda$-small. As $H$ is a quotient of the free product of the groups $G_i$ indexed over a $\lambda$-small cardinal, the definition of regularity implies that $H$ is $\lambda$-small (this is also a consequence of Proposition 1.16 in \cite{Adamek}). Hence, as $H$ is also in $\mathcal{C}$, the factorization of every homomorphism $K\rightarrow H$ through the identity of $H$ proves that $H$ belongs to $\mathcal{C}[\lambda]$.

Now if $H$ is a push-out of a diagram of groups $C\leftarrow A\rightarrow B$ with $A$, $B$ and $C$ $\lambda$-small groups in $\mathcal{C}$, $H$ is in particular a quotient of the free product $B\ast C$, thus $|H|\leq |B\ast C|$. If $B$ and $C$ are finite, $|B\ast C|=\aleph_0< \lambda$, and if not, $|B\ast C|=\textrm{max } (|B|,|C|)<\lambda$, because $|B|$ and $|C|$ are $\lambda$-small. So we obtain the $H$ is $\lambda$-small. As $H$ also belongs to $\mathcal{C}$, we are done.
\end{proof}

Now we can establish closure properties of $\mathcal{C}[\lambda]$. The first is a straightforward consequence of the previous result:

\begin{corollary}

If $\mathcal{C}$ is closed under direct limits, $\mathcal{C}[\lambda]$ is so for every infinite regular cardinal $\lambda$.

\end{corollary}
\begin{proof} Let $H$ be a direct limit of groups in $\mathcal{C}[\lambda]$. By the previous proposition, $H$ is a direct limit of $\lambda$-small groups in $\mathcal{C}$, and, every direct limit of $\lambda$-small groups in $\mathcal{C}$ belongs to $\mathcal{C}[\lambda]$, so we are done.
\end{proof}

Observe that the proof establishes that the $\lambda$-small groups in $\mathcal{C}$ generate $\mathcal{C}[\lambda]$ under colimits.

The following result is the group-theoretical analogue of \cite[Theorem 6.1]{CP04}, although the proof is very different and purely group-theoretical.

\begin{proposition}
\label{lambda-extensions}
If $\C$ is any class of groups closed under direct limits and extensions,
then $\C[\lambda]$ is also closed under extensions.
\end{proposition}
\begin{proof} The idea throughout the proof is to find $\lambda$-small groups that are in $\C$ from the starting short exact sequence. In the end, we will get into (non exact) sequences of small groups in $\C$, take normal closures to make them exact, and pass to the limit to get the desired factorization.
	
Let $1 \to X \to Y \stackrel{\pi}{\to} Z\to 1$ be a short exact sequence with
$X$ and $Z$ in $\C[\lambda]$. We have to prove that $Y\in \C[\lambda]$. Let $j: K \to Y$
be any group homomorphism, with $K$ $\lambda$-small.
Since $Z\in \C[\lambda]$, $\pi j$ factors through some $\lambda$-small group $Z_0$ in $\C$.
Taking the pullback along $Z_0 \to Z$ we obtain a short exact sequence
$1 \to X \cong \Ker \pi_0 \to Y_0 \stackrel{\pi_0}{\to} Z_0 \to 1$.
Let $j_0: K \to Y_0$ be the induced homomorphism and consider
 $K_0$ the subgroup of $Y_0$ generated by $j_0(K)$ and $s(Z_0)$
where $s$ is a chosen section (non necessarily homomorphism) of $\pi_0$.
Let $K'$ be the kernel of $K_0\to Z_0$. Clearly, $K'=K_0\cap X$ and it is $\lambda$-small because $K_0$ is $\lambda$-small.
Since $X\in \mathcal{C}[\lambda]$, the inclusion
$K'\hookrightarrow X$ factors as $K'\hookrightarrow X_0 \to X$ for
some small group $X_0 \in \mathcal{C}$.
$$
\begin{array}{ccccccc}
 &  & K' & \to & X_0 & \to & X\\
 & & \downarrow & & \downarrow && \downarrow\\
K & \to & K_0 & \to & Y_0 & \to & Y\\
 & & \downarrow &  & \downarrowright{\pi_0} && \downarrow\\
&  & Z_0& = & Z_0 & \to & Z
\end{array}
$$
Let now $P_0$ be the push-out of
$ K_0\leftarrow K' \to X_0$, and denote by $(X_0)^N$ the normal closure of $X_0$ in $P_0$. The quotient $P_0/(X_0)^N\cong Z_0$, and
$P_0$ is a $\lambda$-small group in $\mathcal{C}$, because both $K'$ and $X_0$ are $\lambda$-small. Observe that $(X_0)^N$
need not be in $\mathcal{C}$. In any case, it is $\lambda$-small and therefore the inclusion $(X_0)^N\to X$ factors through some $\lambda$-small group $X_1\in \mathcal{C}$. Iterating this process we get an increasing chain of short sequences over $X \to Y \to Z$ which are, alternatively, in $\mathcal{C}$ but not exact ($X_i\to P_i\to Z_0$), and exact but not in $\mathcal{C}$ ($X_i^N\to P_i\to Z_0$). Here the subscript ``$\infty$" denotes the (countable) colimit of each linear diagram of homomorphisms.
$$
\begin{array}{ccccccccccc}
X_0 & \to & (X_0)^N & \to & X_1 & \to ... \to & X_i & \to & (X_i)^N  &  \to ... \to & X_\infty=(X_\infty)^N\\
\downarrow & & \downarrow & & \downarrow && \downarrow && \downarrow && \downarrow\\
P_0 & \to & P_0 & \to & P_1 & \to ... \to & P_i & \to & P_i  & \to ... \to & P_\infty\\
\downarrow & & \downarrow & & \downarrow && \downarrow && \downarrow && \downarrow\\
Z_0 & \to & Z_0 & \to & Z_0 & \to ... \to & Z_0 & \to & Z_0  & \to ... \to & Z_0
\end{array}
$$
In the limit we get a short exact sequence $1\to X_\infty=(X_\infty)^N \to P_\infty \to Z_0 \to 1$, where $X_\infty$ is $\lambda$-small and is in $\mathcal{C}$. As the linear colimit is indexed over a countable ordinal and the cardinal of $P_i$ is $\lambda$-small for every $i$, regularity of $\lambda$ implies that $P_{\infty}$ is $\lambda$-small. Moreover, $P_{\infty}$ is also in $\mathcal{C}$, because $\mathcal{C}$ is closed under extensions. Hence, $K\to Y$ factors through $K\to P_\infty \to Y$, and this shows that $\mathcal{C}[\lambda]$ is closed under extensions.
\end{proof}

\begin{remark}
{\rm
Observe that the hypothesis of $\mathcal{C}$ closed under direct limits can be weakened to closed under linear direct limits, but we preferred the current statement for simplicity.
}
\end{remark}

\begin{definition}
\label{thegenerator}
{\rm Given a class $\C$ closed under direct limits, and an infinite regular cardinal $\lambda$, let $B_{\lambda}$ the free product of a set of representatives of isomorphism classes of $\lambda$-small groups in $\C$.
	}
\end{definition}
We will need the following result for the class $\C=\overline{\mathcal{C}(A)}$:

\begin{lemma}
\label{equalitylambdaB}
Let $\C$ be a class of groups closed under direct limits.
For every infinite regular cardinal, $\C[\lambda]=\mathcal{C}(B_{\lambda})$.

\end{lemma}
\begin{proof} If $K$ is an element of $\C[\lambda]$, it is a direct limit of $\lambda$-small groups in $\C$, by Proposition~\ref{lambda-closed}. But every such $\lambda$-small group is a retract of $B_{\lambda}$, so it belongs to $\mathcal{C}(B_{\lambda})$,
and hence so does $K$ since $\mathcal{C}(B_{\lambda})$ is closed under direct limits.
Conversely, every group $H$ in $\mathcal{C}(B_{\lambda})$ is an (iterated) direct limit of copies of $B_{\lambda}$,
and in particular a direct limit of $\lambda$-small groups, so $H\in \C[\lambda]$ and we are done.
\end{proof}

Now we are ready to prove the main theorem in this section:

\begin{theorem}
\label{maintheorem}
For any group $A$, there exists a group $B$ such that $\overline{\mathcal{C}(A)}=\mathcal{C}(B)$.
\end{theorem}
\begin{proof} Let $A$ be any group, $\lambda$ an infinite regular limit cardinal whose cofinality is bigger than $|A|$. Then, by Lemma~\ref{equalitylambdaB}, we have $\overline{\mathcal{C}(A)}[\lambda]=\mathcal{C}(B)$, where $B=B_{\lambda}$. In particular, the class $\overline{\mathcal{C}(A)}[\lambda]$ is closed under direct limits. The conditions over $\lambda$ imply that $A$ is $\lambda$-small, and hence $\mathcal{C}(A)\subseteq \overline{\mathcal{C}(A)}[\lambda]$. But by Proposition~\ref{lambda-extensions}, $\overline{\mathcal{C}(A)}[\lambda]$ is closed under extensions, so $\overline{\mathcal{C}(A)}\subseteq \overline{\mathcal{C}(A)}[\lambda]$. On the other hand,  $\overline{\mathcal{C}(A)}[\lambda]$ is a subclass of $\overline{\mathcal{C}(A)}$ by definition, and the result follows.
\end{proof}

It is likely that the assumptions over the cardinal may be relaxed, but this not affects the existence or not of the cellular generator.

In next section we will use different cellular generators for concrete computations and in particular, and perhaps more importantly, to define in general the co-reflection associated to the class $\overline{\mathcal{C}(A)}$ for \emph{every} group $A$.

\section{Closed classes}
\label{closedclasses}

\subsection{Classes of groups and associated co-reflections}

Let $A$ be a group. The goal of this section is to analyze the relations between four different classes that contain the group $A$, namely:

\begin{itemize}

\item The \emph{cellular} class $\mathcal{C}(A)$, the closure of $\{A\}$ under direct limits. Their elements are called {\it $A$-cellular groups}.

\item The \emph{socular} class $\mathcal{C}_q(A)$, the closure of $\{A\}$ under direct limits  and quotients. Their elements are called {\it $A$-generated} or {\it $A$-socular groups}.

\item The \emph{radical} class $\mathcal{C}_t(A)$, the closure of $\{A\}$ under direct limits, extensions and quotients. Their elements are called {\it $A$-constructible} or {\it $A$-radical groups}.

\item The \emph{acyclic} class $\overline{\mathcal{C}(A)}$, the closure of $\{A\}$ under direct limits and extensions. Their elements are called {\it $A$-acyclic  groups}.

\end{itemize}

The hierarchy between these classes can be depicted in the following way:
\begin{equation}
\label{hierarchy}
\xymatrix{\mathcal{C}(A) \ar@{^{(}->}[r] \ar@{^{(}->}[d] &  \overline{\mathcal{C}(A)} \ar@{^{(}->}[d] \\
\mathcal{C}_q(A) \ar@{^{(}->}[r] &  \mathcal{C}_t(A) }
\end{equation}
We use in particular the terminology of W. Chach\'olski, who studied the problem in the non singly-generated case in \cite{Cha14}.

It is known that a group $G$ is $A$-cellular if and only if $C_AG=G$ (\cite{RS01}, Section 3). We will check now that similar statements hold for $A$-generated and $A$-constructible groups. Probably this result is known to experts, but we do not know any explicit proof of it. We offer purely algebraic arguments here (see also \cite[Theorem~ 6]{RS00} for part (2)).

\begin{proposition}
\label{ST}
Let $A$ and $G$ be groups. Then:
\begin{enumerate}
\item $G$ is $A$-generated if and only if $S_AG=G$.
\item $G$ is $A$-constructible if and only if $T_AG=G$.
\end{enumerate}
\end{proposition}
\begin{proof}
Let $\C$ be the class of groups $G$ such that $S_AG=G$. Clearly $A\in \C$. Suppose that $H\in \C$, and  $\pi: H \to G$ is an epimorphism. If $\pi(x) \in G/N$, as $S_AG=G$, there exists $\varphi_i:A\to G$, and $a_i\in A$ such that $x=\varphi_1(a_1) \varphi_2(a_2) \cdots \varphi_n(a_n)$. Therefore $\pi(x)= \pi\varphi_1(a_1) \pi\varphi_2(a_2) \cdots \pi\varphi_n(a_n)$, hence $\pi(x)\in S_A(G/N)$, and this ensures that $\C$ is closed under quotients. A similar argument shows that $\C$ is closed under direct limits, and hence $\C_q(A) \subset \C$. Conversely, if $S_A(G)=G$, then $G$ is a quotient of a free product of copies of $A$, so $G\in \C_q(A)$. Thus $\C\subset \C_q(A)$, and we get  $\C_q(A)=\C$.


Let us check the second statement. Consider now the class $\mathcal{C}$ of groups such that $G=T_AG$. Clearly $A$ belongs to $\mathcal{C}$, so let us see that $\mathcal{C}$ is closed under quotients and extensions. First, for every homomorphism of groups $f:H\rightarrow G$ we have $f(T_AH)\subseteq T_AG$. Then, if $H=T_AH$ and $G$ is a quotient of $H$, $G=T_AG$, and hence $G$ belongs to $\mathcal{C}$. On the other hand, assume that $G$ is an extension $G_1\rightarrow G\rightarrow G_2$, with $G_1$ and $G_2$ in $\mathcal{C}$. As $G_1=T_AG_1$, $G_1\subseteq T_AG$. Hence, there is a projection $G_2\rightarrow G/T_AG$. But $G_2=T_AG_2$ and $T_A(G/T_AG)=0$, and thus $G/T_AG=0$ and $G=T_AG.$ We obtain that $\mathcal{C}$ is also closed under extensions, and then $\mathcal{C}_t(A)\subseteq \mathcal{C}$.

Now assume $G=T_AG$. Then $G$ is constructed out of $A$ by means of $A$-socles, extensions and a direct union, and thus, by the previous case and the fact that a direct union is a direct limit, $G$ belongs to the closure of $A$ by direct limits, quotients and extensions. So $\mathcal{C}\subseteq \mathcal{C}_t(A)$, and we are done.
\end{proof}

The previous result provides co-reflections for the respective classes.

\begin{corollary}
In the previous notation, the categories $\mathcal{C}(A)$, $\mathcal{C}_q(A)$ and $\mathcal{C}_t(A)$ are co-reflective. The co-reflections are respectively given by the $A$-cellular cover, the $A$-socle and the $A$-radical. \qed
\end{corollary}
\begin{remark}
{\rm
Observe that the arguments of the proof of the previous proposition are algebraic and quite formal. Hence, they are supposed to work in more general categories, as for example categories of modules.
}
\end{remark}

The next question is immediate: giving a group $A$, is the class $\overline{\mathcal{C}(A)}$ co-reflective? In order to undertake this problem, we introduce one of the main definitions in this context:

\begin{definition}
{\rm

Given a group $A$ and the acyclic class $\overline{\mathcal{C}(A)}$, the co-reflection associated to this class (if it exists) will always be denoted by $D_A$. For a group $G$, $D_AG$ will be called the $A$-\emph{acyclic cover} of $G$.
}
\end{definition}

 We will see in this section that Theorem \ref{maintheorem} gives a general answer to the previous question, and in particular assures the existence of $D_A$ for every group $A$. Before establishing this result, it is worth to point out the partial solution to the question obtained in \cite{RS01}, where the authors assumed the existence of a 2-dimensional Moore space for $A$ in order to construct the co-reflection. Recall that given a group $A$, a \emph{Moore space} of type $M(A,1)$ is a CW-complex $X$ such that $\pi_1X=A$ and $H_iX=0$ for $i>1$:

\begin{theorem}[\cite{RS01}, Theorem 4.3]
\label{DAH}
Let $A$ be the fundamental group of a two-dimensional Moore space.  Then the
inclusion
$\overline{\mathcal{C}(A)}\hookrightarrow Groups$ admits a right adjoint $D_A$.
Furthermore, for each group $H$, there is a central extension
$$
L \hookrightarrow D_A(H) \twoheadrightarrow T_A(H)
$$
such that $\Hom(A_{\rm ab},L)= \Ext(A_{\rm ab},L)=0$, and this extension is
initial with respect to this property.
\end{theorem}

This result has interesting features. First, for $A$ in the conditions of the statement, it gives a manageable way to compute the value of the $A$-acyclic cover of any group $H$, as it is identified with the fundamental group of the homotopy fiber of the $M(A,1)$-nullification of $K(T_AH,1)$, i.e. localization with respect to the constant map over $M(A,1)$ (see \cite{F97}, chapter 1). Moreover, the similarity with Theorem \ref{CAH} should be remarked, and in fact for every group $H$ there exists a natural homomorphism $C_AH\rightarrow D_AH$ which induces a commutative diagram of extensions:

\begin{equation}
\label{diagram}
\xymatrix{K \ar[d] \ar[r] & C_AH \ar[d] \ar[r] & S_AH \ar[d] \\
L  \ar[r] & D_AH \ar[r] & T_AH, }
\end{equation}
where $K$ and $L$ are the respective kernels and the homomorphism $S_AH\rightarrow T_AH$ is injective. Let $M=M(A,1)$ be a two-dimensional Moore space for $A$ and $Y=\textrm{Cof}_MK(S_AH,1)$ be the corresponding Chach\'olski homotopy cofiber  (\ref{cof1}). Then it is stated in the proof of Theorem 3.7 in \cite{RS01} that $K=\pi_2P_{\Sigma M}Y$. Analogously, let $Y'$ be $\textrm{Cof}_MK(T_AH,1)$; then, the proof of Theorem 4.3 in the mentioned article proves that $L=\pi_2P_{M}Y'$. This result will be useful in Proposition \ref{counterexample} below.

Remark that the co-reflection $D_A$ is built in Theorem \ref{DAH} in a homotopical (non-algebraic) way. This homotopical nature of the construction made it impossible for these authors to define $D_A$ for a general $A$, as the definition depended on the existence of a two-dimensional Moore space for the group. In the sequel we will avoid the difficulty by defining $D_A$ as a cellular cover in a pure group-theoretical way.

We present now our general construction of the co-reflection $D_A$ for every group $A$.

\begin{proposition}
Let $A$ be a group, and $B$ a cellular generator of the class $\overline{\mathcal{C}(A)}$. Then a group $G$ is $A$-acyclic if and only if $C_BG=G$.
\end{proposition}
\begin{proof} It is straightforward from Theorem~\ref{maintheorem}.
\end{proof}

\begin{corollary}
\label{corocorre}

For every group $A$, the class $\overline{\mathcal{C}(A)}$ is co-reflective, and the co-reflection $D_A$ is given by $C_B$.

\end{corollary}

\begin{remark}\rm{
Observe in particular that, although Theorem \ref{DAH} above proves the existence of $D_A$ for some instances of $A$ and provides a way to compute the acyclic cover, it does not describe a cellular generator of the class. 
}
\end{remark}

After establishing the existence of the co-reflection $D_A$, we conclude this subsection by stating some properties of the co-reflection $D_A$ and the cellular generators of the acyclic class.

\begin{proposition}
\label{uniqueness}
If $B$ and $B'$ are cellular generators of the class $\overline{\mathcal{C}(A)}$ and $H$ is group, the cellular covers of $H$ with respect to $B$ and $B'$ are isomorphic.
\end{proposition}
\begin{proof}
As $B$ and $B'$ are cellular generators of the acyclic class, $C_B=D_A$ and $C_{B'}=D_A$.
\end{proof}

\begin{lemma}
\label{SyT}
Let $A$ be a group, $B$ a cellular generator of $\overline{\mathcal{C}(A)}$. Then $S_B=T_A$.

\end{lemma}
\begin{proof}
Both $S_B$ and $T_A$ are co-reflections over the class $\mathcal{C}_t(A)$. The result follows from uniqueness of co-reflections.
\end{proof}


\begin{corollary}
For every group $A$, we have $D_A=D_AT_A$.
\end{corollary}
\begin{proof}
It is well-known that for every group $A$, $C_A=C_AS_A$. Now, if $B$ is a cellular generator of $\mathcal{C}(A)$, $C_B=D_A$ and $S_B=T_A$. So we are done.
\end{proof}

In particular, in the computation of $D_AH$ for a group $H$ it is enough to take as input the $A$-radical of $H$.

The following proposition provides concrete descriptions of $D_A$.

\begin{proposition}
Let $H$ and $G$ be finite non-trivial groups, with $H$ simple, $|H|\lneq |G|$ and such that there is exactly one subgroup of $G$ that is isomorphic to $H$. Then $D_HG=H$, and $G/H$ does not belong to $\overline{\mathcal{C}(H)}$.
\end{proposition}
\begin{proof} As $H$ is simple and normal in $G$, $S_HG=T_HG=H$. Thus, $$D_HG=D_HT_HG=D_HH=H,$$ as $H$ is clearly $H$-acyclic. Moreover, if $G/H$ belonged to $\overline{\mathcal{C}(H)}$, then $G$  would also belong, and that is impossible because $D_HG\neq G$.
\end{proof}

 Note that this result gives information for classes $\overline{\mathcal{C}(A)}$ with $H_2A\neq 0$. For example, if $H=A_n$, $G=\Sigma_n$ and $n\geq 5$, the proposition implies that $D_{A_n}\Sigma_n=A_n$, and hence neither $\Sigma_n$ nor $\mathbb{Z}/2$ can be constructed out of $A_n$ by direct limits and extensions. The groups for which $H_2A\neq 0$ are beyond the scope of Theorem \ref{DAH}, as the existence of the Moore space implies that $H_2A$ should be trivial there.

\begin{remark}
\rm{
We note that the use of the word ``acyclic" in our context comes from Homotopy Theory, and especially from Bousfield, who thoroughly studied the classes of spaces for which a homology theory vanishes. We would like to emphasize that the acyclic cover owes its name to the fact that, in the cases studied by Rodr\'{i}guez-Scherer in \cite{RS01}, it is the fundamental group of the homotopy fibre of a localization; in general this group may have non-trivial ordinary homology, so it is not acyclic in the usual sense of Group Homology. 
}
\end{remark}

\subsection{Comparing the classes}

From now on we will concentrate in measuring the difference between the classes above. Fix a group $A$, and let $F$ be one of the four co-reflections. Then given a group $G$, $FG$ measures in general how close the group $G$ is from belonging to the corresponding class. Moreover, the homomorphisms $S_AG\rightarrow G$, $T_AG\rightarrow G$ and $S_AG\rightarrow T_AG$ are always injective, and it is a consequence of \cite[Proposition 3.1]{RS01} and Corollary \ref{corocorre} above that $C_AG\rightarrow S_AG$ and $D_AG\rightarrow T_AG$ are always surjective. Hence, we have a commutative diagram
\begin{equation}
\label{hierarchy2}
\xymatrix{C_AG \ar@{>>}[r] \ar[d] & S_AG  \ar@{^{(}->}[d] \\
D_AG \ar@{>>}[r] &  T_AG }
\end{equation} which generalizes the diagram~(\ref{diagram}) above to the case of a general $A$. Note that in general, $C_AG\rightarrow D_AG$ is neither surjective nor injective.

Using an example, we start by showing that the four inclusions of classes in the diagram~(\ref{hierarchy}) can be strict. Of course, the previous results imply that two such classes are identical if and only if their associated co-reflections are the same.

 \begin{example}

  \label{alldiferent}
  {\rm Consider $G$ the direct product of the Mathieu group $M_{12}$ with the cyclic group $ \Z /9$. We use Proposition~\ref{coreflection-product} to compute the value of the co-reflections over this group. As the socle and the radical are always normal subgroups and $M_{12}$ is simple and contains elements of order 3, we have $S_{\Z/3}G=M_{12}\times \Z/3$ and $T_{\Z/3}G=G$.

  Moreover, taking into account that $H_2(M_{12})=\Z /2$ (see for example Section 11 in \cite{BCFS13})
  we have that $C_{\Z /3}M_{12}=\widetilde M_{12}$, the universal cover of $M_{12}$. Observe that the alternative $C_{\Z /3}M_{12}=\Z /2\times M_{12}$ is not possible as $\Z/2\times M_{12}$ is not $\Z /3$-cellular. Hence, $C_{\Z /3}G=\widetilde M_{12}\times \Z/3$.
On the other hand, $D_{\Z /3}G= \widetilde M_{12} \times \Z/9$, as $G$ is $\Z /3$-constructible.

}
\end{example}

Next we will discuss when the socular class $\mathcal{C}_q(A)$ is closed under extensions, and hence equal to the radical class $\mathcal{C}_t(A)$. For example, if $A=\mathbb{Z}/p$ for some prime $p$,
it is known that there are many groups $G$ such that $G/S_{\Z/p}G$ has again $p$-torsion, and same happens if we change $\mathbb{Z}/p$ by $\ast_{j\geq 1}\mathbb{Z}/{p^j}$; this contrasts with the case of abelian groups, in which for $A=\bigoplus_{j\geq 1}\mathbb{Z}/{p^j}$, $T_AN=S_AN$ for every abelian group.

When $A$ is the additive group of the rational numbers, the situation is similar. If $N$ is abelian, $S_{\mathbb{Q}}N=T_{\mathbb{Q}}N$, because the image of every homomorphism $\mathbb{Q}\rightarrow N$ is divisible and then splits out of $N$. However, if $N$ is not abelian the equality is not true in general. We have found no example of this situation in the literature, so we describe one in the sequel:

\begin{example}
\label{QN}
{\rm
Let $G$ be the push-out of inclusions $\mathbb{Q}\hookleftarrow \mathbb{Z} \hookrightarrow \mathbb{Z}[1/p]$.  Following (\cite{Joh97}, ch. 5), a presentation of the group $(\mathbb{Q},+)$ is given by:

$$\{x_1, x_2, x_3, \ldots \textrm{ } | \textrm{ } x_n^n=x_{n-1} \textrm{ for } n\geq 2 \}.$$

If we look at the rationals as a subgroup of the reals, the fraction $\frac{1}{m!}$ corresponds to the generator $x_m$, for every $m\geq 1$.

Analogously, a presentation of $\Z [1/p]$ is:
$$\langle y_1, y_2, y_3, \ldots \textrm{ } | \textrm{ } y_n^p=y_{n-1}, \textrm{ for } n\geq 2 \rangle.$$

Now identifying $\mathbb{Z}[1/p]$ with the irreducible fractions of $\mathbb{R}$ with denominator a power of $p$, the fraction $\frac{1}{p^m}$ corresponds to $y_m$.

Using Seifert-van Kampen theorem, we obtain that a presentation for the group $G$ is given by:
$$\langle x_1, y_1, x_2, y_2, \ldots \textrm{ } | \textrm{ } x_n^n=x_{n-1}, y_n^p=y_{n-1}, x_1=y_1, \textrm{ for } n\geq 2 \rangle.$$

Note that the amalgamation is reflected in the relation $x_1=y_1$.

Consider now the normal subgroup $N$ of $G$ generated by the copy of $\mathbb{Q}$ inside $G$ given by the generators $\{x_i\}$. This group is generated by conjugates of this copy of $\mathbb{Q}$, and therefore $N<S_{\mathbb{Q}}G$. On the other hand, as $G$ is torsion-free, every homomorphism $\phi:\mathbb{Q}\rightarrow G$ is injective, and then $\textrm{Im }\phi$ is isomorphic to $\mathbb{Q}$. Now invoking a result of Moldavanskii (\cite{Mol67}, see also \cite{KS70}, page 229), and taking into account that $\mathbb{Q}\nless\Z [1/p]$, we obtain that $\textrm{Im }\phi<N$ for every $\phi:\mathbb{Q}\rightarrow G$, and this already implies that $N=S_{\mathbb{Q}}G$.

Next we will describe the quotient $G/N$. A presentation of this group is obtained by trivializing the generators $\{x_i\}$ in the previous presentation of $G$ (which in particular implies trivializing $y_1$); so we get (after a renumbering of generators):
$$\langle y_1, y_2, \ldots \textrm{ } | \textrm{ } y_1^p=1, y_n^p=y_{n-1}, \textrm{ for } n\geq 2 \rangle.$$

This is a classical presentation of the quasicyclic group $\mathbb{Z}/p^{\infty}$, being the generator of every $\Z /p^n$ identified with $y_n$. As there exists an epimorphism $\mathbb{Q}\rightarrow \mathbb{Z}/p^{\infty}$, $S_{\mathbb{Q}}G$ is strictly contained in $T_{\mathbb{Q}}G$, as claimed. In fact,  $G=T_{\mathbb{Q}}G$, as the only nontrivial quotients of the quasicyclic group are isomorphic to it.
}
\end{example}

 Using amalgams $\mathbb{Z}[1/q]\hookleftarrow\Z \hookrightarrow \mathbb{Z}[1/p]$, with $q\neq p$, and taking into account that there are always non-trivial homomorphisms $\Z[1/q]\rightarrow \Z /p^{\infty}$, it can be stated in an analogous way that $S_{\mathbb{Z}[1/q]}\neq T_{\mathbb{Z}[1/q]}$ for every prime $q$; and moreover, similar arguments can be used to prove that for a non-trivial set of primes $J$, the co-reflection $S_{\mathbb{Z}[J^{-1}]}$ is never equal to $T_{\mathbb{Z}[J^{-1}]}$. In particular we obtain:

\begin{proposition}
\label{additive}
If $A$ is a non-trivial additive subgroup of $\mathbb{Q}$, $S_A=T_A$ if and only if $A=\mathbb{Z}$.
\end{proposition}
\begin{proof}
By work of Beaumont-Zuckerman \cite{BeZu51}, the additive subgroups of $\mathbb{Q}$ are the rational subgroups of rank one. As stated in Example 6.2 of \cite{RS01} (see also \cite{FOW83}), the only non-trivial subgroups $A$ among them for which $S_A$ can be equal to $T_A$ are those of the shape $A=\mathbb{Z}[J^{-1}]$, for a set of primes $J$. But if $J$ is non-empty and $p\in J$, the same argument of Example \ref{QN} above proves that the $\mathbb{Z}[J^{-1}]$-socle of the push-out $\mathbb{Z}[J^{-1}]\hookleftarrow \Z \hookrightarrow \Z [1/p]$ is strictly contained in its $\mathbb{Z}[J^{-1}]$-radical, in turn equal to the whole group. On the other hand, if $J=\emptyset$, every group coincide with its $\mathbb{Z}$-socle, and then $S_{\Z}=T_{\Z}$. So we are done.
\end{proof}

The moral here is that it is not easy to find examples of groups $A$ for which $S_A=T_A$. However, there are at least two ways of constructing such examples. The first family arises in the context of varieties of groups. Consider a set of equations, and the verbal subgroup $W$ defined by these equations in the free group $F_{\infty}$. It is proved in \cite{CRS99} that for every $W$ there exists a locally free group $A$ such that the homomorphisms of $W$ into any group $G$ identify the $W$-perfect radical, and in particular, this fact guarantees that $S_A=T_A$. The inspiring (and first) ``generator" of this kind was the group constructed by Berrick-Casacuberta in  \cite{BC99}, and a closely related group will be the crucial ingredient of Proposition \ref{counterexample} below.

The second family of examples can be constructed as a byproduct of Theorem \ref{maintheorem}, taking account of the following two results:

\begin{proposition}
\label{TASA}
Suppose that for a group $A$ the class $\C(A)$ is closed under extensions, i.e. $\mathcal{C}(A)=\overline{\mathcal{C}(A)}$. Then $\mathcal{C}_q(A)=\mathcal{C}_t(A)$.
\end{proposition}
\begin{proof} In this case, for every group $G$, $C_AG=D_AG$, and then there is a diagram
$$
\xymatrix{D_AG \ar@{>>}[r] \ar[d]^{Id} & S_AG  \ar[d]^{j} \\
D_AG \ar@{>>}[r] &  T_AG }
$$
As $j$ should be an epimorphism, and it is always a monomorphism, it is an isomorphism.
\end{proof}

\begin{corollary}

Let $A$ be a group. Then for any cellular generator $B$ of the class $\overline{\mathcal{C}(A)}$ we have $S_B=T_B$.

\end{corollary}

Recall that by Theorem \ref{maintheorem}, it is always possible to construct such $B$ out of any given $A$, and this provides many examples of (big) groups for which $S_A=T_A$. It is natural and interesting to ask if there is a converse of the previous proposition, so we consider the following problem:

\begin{question}
\label{SoT}
If $T_A = S_A$, for a certain group $A$, is it true that $\mathcal{C}(A)=\overline{\mathcal{C}(A)}$?
\end{question}

We will see in next section that the general answer is negative. Before, we will discuss in which cases such an answer can be positive. The following result is easy to prove and gives a sufficient condition.

\begin{proposition}
Let $A$ be a group, and assume that $\mathcal{C}(A)=\mathcal{C}_q(A)$ and $S_A=T_A$. Then $\mathcal{C}(A)=\overline{\mathcal{C}(A)}$.
\end{proposition}
\begin{proof} As the members of the classes are defined precisely as the groups for which the corresponding co-reflection is the identity, $S_A=T_A$ implies $\mathcal{C}_q(A)=\mathcal{C}_t(A)$, and then by hypothesis $\mathcal{C}(A)=\mathcal{C}_t(A)$. But the inclusions $\mathcal{C}(A)\subseteq \overline{\mathcal{C}(A)} \subseteq \mathcal{C}_t(A)$ always hold, so the result follows.
\end{proof}

It can be seen that the condition of the proposition holds for the universal group of Berrick-Casacuberta. However, the following example shows that the hypotheses of the previous proposition are not necessary in general for cellular generators of closures under extensions:

\begin{example}
{\rm
Given a prime $p$, consider a cellular generator $B$ of the class $\overline{\mathcal{C}(\mathbb{Z}/p)}$. If $\mathcal{C}(B)=\mathcal{C}_q(B)$, this would imply that $C_B=S_B$, and equivalently $D_{\mathbb{Z}/p}=T_{\mathbb{Z}/p}$. But this is not true in general, as can be deduced for instance from Example \ref{alldiferent}.
Observe that we obtain in particular that $\mathcal{C}(B)=\overline{\mathcal{C}(B)}\neq \mathcal{C}_q(B)=\mathcal{C}_t(B)$ in this case.
}
\end{example}

In general, the relation between the group co-reflections $C_A$ and $C_B$ ($=D_A$) has been investigated when $H_2A=0$, and it is quite well understood when $A$ is a cyclic group or a subring of $\mathbb{Q}$ (see for example \cite{RS01}). As seen above, the latter is not useful in our context, as $S_A\neq T_A$ for these groups. However, in next section, we will take profit of some features of the case $H_2A=0$ to describe a counterexample for Question \ref{SoT}. We emphasize that in this proof we will not need the results of Section \ref{generating}.

\section{The Burnside radical and a free Burnside group}
\label{Burnside}

 Recall from \cite{CRS99} that the \emph{Burnside idempotent radical} for a prime $p$ of a group $G$ is the largest subgroup $H<G$ which is generated by $p$-powers of its elements. As it is established in (\cite{CRS99}, Theorem 3.3), given any group $G$ and a prime $p$, its corresponding Burnside radical can be always generated by the images of homomorphisms of a group $F$, called the \emph{generator} of the radical. This group will be the key object in the proof of Proposition \ref{counterexample} below, so we will compile next some properties of it that will be useful in the sequel, in particular in the homology computations of Lemma \ref{homology}.

\begin{proposition}

The following statements hold for the group $F$:

\begin{enumerate}

\item $F$ generates the Burnside radical.

\item $F$ is an infinite free product of groups $F_{\mathbf{n}}$, where  $\mathbf{n}$ runs over all the non-decreasing sequences of integers.

\item The groups $F_{\mathbf{n}}$ are countable.

\item Every $F_{\mathbf{n}}$ is a directed system of finitely generated free groups $\{F_{\mathbf{n},r}\}$, and in particular is locally free and torsion-free.

\item For every $\mathbf{n}$ and $r<s$, $\textrm{rank }F_{\mathbf{n},r}<\textrm{rank }F_{\mathbf{n},s}$.

\item The image of every bonding homomorphism $\phi_r:F_{\mathbf{n},r}\rightarrow F_{\mathbf{n},r+1}$ is always contained in a product of $p$-powers of generators of $F_{\mathbf{n},r+1}$.

\item For every $\mathbf{n}$ and $r$, there exist distinguished systems $\{x_{{\mathbf{n}},r}\}$ of generators of $F_{\mathbf{n},r}$, such that for every $i\neq j$, $\phi_r(x_{\mathbf{i},r})$ and $\phi_r(x_{\mathbf{j},r})$ are spanned by disjoint subsets of the system $\{x_{\mathbf{n},r+1}\}$.

\item $F$ is locally free and torsion-free.
\end{enumerate}
\end{proposition}

\begin{proof}
Statement 1) is exactly Theorem 3.3 in \cite{CRS99}, and the remaining ones are established in the course of the proof of this result. We adopt the notation from there. The definition of $F$ as a free product of 2) appears in the second paragraph of that proof. Just after it, the construction of $F_{\mathbf{n}}$ as a directed system of free groups $F_{\mathbf{n},r}$ is undertaken, with a precise description of the generator of each group and the bonding homomorphisms $\phi$. The fact that the groups $F_{\mathbf{n},r}$ are finitely generated established there implies 3) and the first part of 4), and the construction as a directed system of free groups gives locally freeness and torsion-freeness of $F_{\mathbf{n}}$ (rest of statement 4). The indexation of the generators in the second paragraph implies the inequalities of 5), and the fact that in this case the only non-trivial word that defines the variety is $w=x^p$ gives 6), by the definition of the product in that paragraph. As all the symbols that appear in the definition of $x_r$ appear again in the generators that span $\phi_r(x_r)$ we have 7). Finally, 8) is a straight consequence of 4), as being locally free is closed under free products.
\end{proof}

Now we are ready to compute the ordinary homology of $F$.

\begin{lemma}
\label{homology}

The first homology group $H_1(F)$ is a free module over $\Zup$ of infinite rank, and $H_iF=0$ if $i>1$.

\end{lemma}
\begin{proof}

We first compute the homology of every $F_{\mathbf{n}}$. As these groups are filtered colimits of free groups, $H_iF_{\mathbf{n}}=0$ if $i>1$. Moreover, the construction of $F_{\mathbf{n}}$ in the mentioned paper as a telescope of homomorphisms that sends generators to products of $p$-powers (see statement 6 of the previous proposition) implies that the first homology group is the direct limit of a system $$\Z\rightarrow \bigoplus_{I_1} \Z\rightarrow \bigoplus_{I_2} \Z\rightarrow \ldots,$$ where:
\begin{itemize}
\item Every direct sum is taken over a finite set of indexes (statement 4 above).
\item If $j<k$, then $I_j<I_k$ (statement 5).
\item For every $k$, the homomorphism $\bigoplus_{I_{k}} \Z\rightarrow \bigoplus_{I_{k+1}}\Z$ takes the $|I_k|$ components of the left group to the first $|I_k|$ of the right group in their order, and the homomorphism is, up to a change of base, multiplication by $p$ (statements 6 and 7).
\end{itemize}

Hence, the direct system that defines $H_1{F_{\mathbf{n}}}$ is a countable direct sum of copies of the system $\mathbb{Z}\stackrel{p}{\rightarrow} \Z \stackrel{p}{\rightarrow}\mathbb{Z}\stackrel{p}{\rightarrow}\ldots$ and then this homology group is a free module over $\mathbb{Z} [1/p]$.
Finally, writing $F$ as the colimit (under the inclusion) of all the subgroups of the shape $F_{\mathbf{i_1}}\ast\ldots \ast F_{\mathbf{i_n}}$ and taking the corresponding limit in homology, the result follows. Note that the rank of $H_1(F)$ as a free module over $\Zup$ is uncountable.
\end{proof}

Now given a prime $q>10^{75}$, $q\neq p$, consider the \emph{free Burnside group} $B(2,q)$ in two generators. It is a group of exponent $q$, and by Corollary 31.2 in \cite{Ol91} the bound $q>10^{75}$ implies that its second homology group is free abelian in a countable number of generators. In the sequel we will denote $B(2,q)$ by $G$.

\begin{proposition}
\label{counterexample}
For the group $F$ above defined, the equality $S_F=T_F$ holds, but $C_F\neq D_F$.
\end{proposition}
\begin{proof}
As $F$ generates the idempotent radical associated to a variety of groups, $S_F=T_F$. We will check now that $C_F\neq D_F$. As $G$ has exponent $q$ and $q$ is coprime to $p$, multiplication by $p$ is an isomorphism. Then $G$ is equal to its idempotent Burnside radical, and then $G=S_FG=T_FG$. Hence, to prove that $C_FG\neq D_FG$ it is enough to check that the kernels of $D_FG\rightarrow G$ and $C_FG\rightarrow G$ are not isomorphic.

 We first compute the kernel of the epimorphism $D_FG\rightarrow G$. Since $F$ is locally free, there exists a 2-dimensional Moore space $M(F,1)$ and therefore, we can use Theorem \ref{DAH} directly to compute such kernel. Consider the Chach\'olski cofibration  \begin{equation}
\label{cofib}
 \bigvee_{[M,K(G,1)]_*} M\rightarrow K(G,1)\to Y,
 \end{equation}
  where $Y$ denotes the homotopy cofiber. We proved just after the mentioned theorem that the kernel of $D_FG\rightarrow G$ is isomorphic to $\pi_2(P_MY)$, which in turn is the Ext-$p$-completion of the group $\pi_2Y$ in the primes for which $F_{\rm ab}$ is uniquely $p$-divisible (\cite{Bou97}, Theorem 7.5). By Lemma \ref{homology}, $F_{\rm ab}=H_1F$ is an $\Zup$-free module, and hence we only need to complete in the prime $p$. Applying Mayer-Vietoris to (\ref{cofib}), and the fact that $\pi(Y)=H_2(Y)$, as $Y$ is simply-connected, we obtain an extension
\begin{equation}
\label{extension}
 A \hookrightarrow \pi_2Y \stackrel{f}{\twoheadrightarrow} R.
 \end{equation}
 Here $A=H_2(G)$ is free abelian in a countable number of generators, and $R$ is the kernel of the homomorphism $H_1(\bigvee_{[M,K(G,1)]_*} M)\rightarrow H_1K(G,1)$. In this case $R=\bigoplus_J \Zup$ is a free $\Zup$-module, for some set $J$ of indexes, because $H_1(\bigvee_{[M,K(G,1)]_*} M)$ is so and $\Zup$ is a principal ideal domain. Taking now account of the fact that $\textrm{Hom}(\Z /p^{\infty},\Z)=0$ and (\cite{BK72},VI.2.5), we obtain another extension:
$$\textrm{Ext}(\Z /p^{\infty},A)\hookrightarrow \textrm{Ext}(\Z/p^{\infty},\pi_2Y)\twoheadrightarrow \textrm{Ext}(\Z /p^{\infty},R).$$
 As the Ext-$p$-completion of a free $\Zup$-module is trivial, the term in the right is so, and the homomorphism in the left is an isomorphism. Now observe that again by (\cite{BK72},VI.2.5), the extension $\Z\hookrightarrow A\twoheadrightarrow A$ given by the inclusion of a factor induces another extension $$\textrm{Ext}(\Z /p^{\infty},\Z)\hookrightarrow \textrm{Ext}(\Z /p^{\infty},A)\twoheadrightarrow \textrm{Ext}(\Z /p^{\infty},A).$$ As $\textrm{Ext}(\Z /p^{\infty},\Z)=\Z^{\wedge}_p$
 we obtain that the kernel of $D_FG\rightarrow G$ contains at least a copy of the $p$-adic integers.

Next we will check that $C_FG\neq D_FG$, and in order to do, we will prove that the kernel of $C_FG\rightarrow G$ cannot contain a subgroup isomorphic to $\Z^{\wedge}_p$ . Again from the argument after Theorem~\ref{DAH} we have that this kernel is isomorphic to $\pi_2(P_{\Sigma M}Y)$, and this
is the quotient of $\pi_2Y$ by its $F$-radical $T$, which is in turn equal to the $\Zup$-radical as $\pi_2Y$ is abelian (\cite{CRS99}, Section 3).

Let us describe $T$ and $(\pi_2Y)/T$ more explicitly.
For shortness, denote by $S=S_{\Zup}\pi_2Y$, the ${\Zup}$-socle of $\pi_2Y$.

Consider all the rank 1 $\Z [1/p]$-submodules of $R$ such that the inclusion $\Z [1/p]\hookrightarrow R$ lifts to $\pi_2Y.$ After a change of basis, we may assume that there is a set $I<J$ (that can be empty) such that a) the module $R$ decomposes as a direct sum $R=R_I\oplus R_{J\setminus I}$, being $R_I=\bigoplus_I \Zup$ and $R_{J\setminus I}=\bigoplus_{J\setminus I} \Zup$, and b) every homomorphism $\Z [1/p]\rightarrow R_I$ lifts to $\pi_2Y$.

 Next we will prove $S\cong R_I$, via the restriction of $f$ in (\ref{extension}). First, observe that the elements of $S$ are precisely the elements of $\pi_2Y$ that are uniquely $p$-divisible, and every such element is in the image of some homomorphism $\Zup\rightarrow\pi_2Y$. Now for every homomorphism $g:\Zup\rightarrow \pi_2Y$ the composition $f\circ g$ is again nontrivial: see for example \cite[3F.11]{Hat02}, taking into account that $A$ does not contain $p$-divisible elements. In particular, $g$ is a lifting of a homomorphism $\Zup\rightarrow R_{I}$. Hence, the restriction of $f$ to the socle takes values on $R_I$. Moreover, if $x\in R_I$, $x$ is in the image of some homomorphism $\Zup\rightarrow R_I$, so lifting this homomorphism we obtain a preimage of $x$ in the socle, and $f|_{S}$ is then surjective.  To check injectivity, observe that fixed a homomorphism $\Zup\rightarrow R_I$, the lifting to the socle is unique, because again by \cite[3F.11]{Hat02} the induced homomorphism $\textrm{Hom}(\Zup,\pi_2Y)\rightarrow \textrm{Hom}(\Zup, R_I)$ is injective. Then, the only homomorphism $\Zup\rightarrow S$ that lifts the trivial one is trivial, and then $f(x)=0$ implies $x=0$ for $x\in S$. So we have proved that $f$ maps isomorphically the socle $S$ onto $R_I$. Now we have a diagram, where rows and columns are exact:
$$\xymatrix{ \{1\} \ar[r] \ar[d] &  S \ar[r]^{\simeq} \ar[d] & R_I \ar[d] \\
A \ar[r] \ar[d]^{\simeq} & \pi_2Y \ar[r]_f \ar[d] &  R \ar[d] \\
A \ar[r] & (\pi_2Y)/S \ar[r] & R_{J\setminus I}. \\
}$$

Observe that $(\pi_2Y)/S$ has no uniquely $p$-divisible elements; otherwise, some homomorphism $\Zup\rightarrow R_{J\setminus I}$ could be lifted to that group, and this is not possible by the definition of $I$. This implies that $S=T$. Now it is immediate that $(\pi_2Y)/T$ cannot contain a copy of $\Z^{\wedge}_p$, taking account for example that the $p$-adic integers admit non-trivial homomorphisms from $\Z[1/q]$ ($q$ a prime different from $p$) and $(\pi_2Y)/T$ does not. Hence $D_FG\neq C_FG$ and we are done.
\end{proof}

\begin{remark}
{\rm In the previous proof, we do not know if $(\pi_2Y)/T$ is a free abelian group, because we are unable to show if the aforementioned extension that defines $\pi_2Y$ splits or not. Recall that $\textrm{Ext}(\Zup,\Z)=\Z^{\wedge}_p/\Z$ is uncountable, so there is an uncountable number of abelian extensions of $\Zup$ by $\Z$, but only the trivial extension of $\Z$ by $\Zup$ splits. }
\end{remark}

We conclude by pointing out that the arguments of Proposition \ref{counterexample} are also valid when approaching $C_F$ and $D_F$ for the perfect groups defined by Berrick-Miller in \cite{BeMi92}, for which the Schur multiplier is free abelian. Moreover, \emph{mutatis mutandis}, they may also be used to describe the $\Z /p$-acyclic cover of these Burnside groups of high exponent. This is related with recent results about cellular covers of Burnside groups (\cite{Go12}, \cite{HPR17} \cite{FlMu19}, \cite{FlSc18}), that in particular have been useful to solve some conjectures by Farjoun \cite{Ch08}. On the other hand, note that we have approached the description of the acyclic and cellular covers of a non-nilpotent group with respect to an uncountable non-nilpotent group; these kind of problems have been widely studied in the recent literature, but usually under restrictive hypotheses of finiteness or nilpotency on the groups involved. See for example \cite{BCFS13}, \cite{FGS07}, \cite{FGSS07}, \cite{Fu11}.

\vspace{.7cm}

\textbf{Acknowledgments.} We warmly thank J\'{e}r\^{o}me Scherer for several discussions about the paper and for his thorough revision of the manuscript, which has led to a substantial improvement of the presentation. We also thank Pedro Guil and Fernando Muro for useful discussions. The first author expresses his gratitude to the Department of Mathematics of the University of Almer\'{i}a for their kind hospitality. We thank the referee for his/her detailed and sharp report.


\vskip 0.5 cm

\setlength{\baselineskip}{0.6cm}

Ram\'{o}n Flores

Departamento de Geometr\'{i}a y Topolog\'{i}a, Universidad de Sevilla-IMUS

E-41012 Sevilla, Spain, e-mail: {\tt ramonjflores@us.es}

\bigskip\noindent

Jos\'e L. Rodr\'{\i}guez

Departamento de Matem\'aticas,
Universidad de Almer\'{\i}a

E-04120 Almer\'{\i}a,
Spain, e-mail: {\tt
jlrodri@ual.es}

\end{document}